\documentclass[10pt,a4paper]{article}       

\usepackage[slantedGreek,noBBpl]{mathpazo}
\usepackage{amsmath,amssymb,amsthm}
\usepackage[all]{xy} 
\usepackage{enumerate}
\usepackage[pdftex]{hyperref}
\usepackage{url}

\newtheorem{thm}{Theorem}
\newtheorem{prop}{Proposition}
\newtheorem{lem}{Lemma}

\newtheorem{dfn}{Definition}

\newtheorem{examples}{Example}
\newtheorem{remark}{Remark}[section]
\numberwithin{equation}{section}
\numberwithin{figure}{section}

\makeatletter
\let\c@assumption=\c@thm
\let\c@prop=\c@thm
\let\c@lem=\c@thm
\let\c@cor=\c@thm
\let\c@dfn=\c@thm
\let\c@examples=\c@thm
\makeatother

\newcommand{\squashup}[0]{}
\renewcommand{\epsilon}[0]{\varepsilon}
\renewcommand{\phi}[0]{\varphi}

\newcommand{\A}{\mathcal A}  
\newcommand{\C}{\mathbb  C}  
\newcommand{\F}{\mathcal F}  
\renewcommand{\H}{\mathrm H}  
\renewcommand{\P}{\mathbb  P}  
\newcommand{\R}{\mathbb  R}  
\newcommand{\T}{\mathbb  T}  
\newcommand{\V}{\mathbb  V}  
\newcommand{\X}{\mathcal X}  
\newcommand{\Y}{\mathcal Y}  
\newcommand{\Z}{\mathbb  Z}  

\newcommand{\g}{\mathfrak g}  
\renewcommand{\i}{\mathtt   i}  
\renewcommand{\t}{\mathfrak t}  

\renewcommand{\emph}[1]{\textbf{#1}}

\newcommand{\etale}[0]{\'{e}tale} 
\newcommand{\sq}[1]{/\!\!/_{\!#1}} 
\newcommand{\acts}[0]{\curvearrowright} 
\newcommand{\dual}[0]{\vee} 

\newcommand{\rd}[0]{\partial} 

\renewcommand{\vec}[1]{{#1}} 
\newcommand{\pair}[1]{\left\langle#1\right\rangle} 
\newcommand{\lm}[0]{\mu^{-1}(\tau)} 

\newcommand{\TC}[0]{\mathbb{T}_\mathbb{C}} 
\newcommand{\DG}[0]{\mathrm{DG}} 

\newcommand{\id}[0]{\mathrm{id}} 

\DeclareMathOperator{\Hom}{Hom}
\DeclareMathOperator{\Mor}{Mor}
\DeclareMathOperator{\Ext}{Ext}
\DeclareMathOperator{\src}{src}
\DeclareMathOperator{\tgt}{tgt}

\DeclareMathOperator{\Lie}{Lie}
\DeclareMathOperator{\grad}{grad}
\DeclareMathOperator{\coker}{coker}
\DeclareMathOperator{\codim}{codim}
\DeclareMathOperator{\Cone}{Cone}
\DeclareMathOperator{\Vect}{Vect}

\newcommand{\Diff}{\mathfrak{Diff}}

\renewcommand{\1}{{{\mathchoice {\rm 1\mskip-4mu l} {\rm 1\mskip-4mu l}{\rm 1\mskip-4.5mu l} {\rm 1\mskip-5mu l}}}}
\newcommand{\org}[1]{}

\usepackage{layout}

\begin{document}

\title{The symplectic Deligne-Mumford stack associated to a stacky polytope}
\author{Hironori Sakai}
\date{}

\maketitle

\begin{abstract}
 We discuss a symplectic counterpart of the theory of stacky fans. First, we define a stacky polytope and construct the symplectic Deligne--Mumford stack associated to the stacky polytope. Then we establish a relation between stacky polytopes and stacky fans: the stack associated to a stacky polytope is equivalent to the stack associated to a stacky fan if the stacky fan corresponds to the stacky polytope.
\end{abstract}

\section{Introduction}
Lerman and Malkin introduced a symplectic structure on a Deligne--Mumford stack. They also define a Hamiltonian group actions on a symplectic Deligne--Mumford stack \cite{0908.0903}. The motivation of their work is based on the following proposal: using stacks is preferable to using ordinary orbifold atlases when we study geometries of orbifolds.

The theory of stacks has been developed by algebraic geometers. Abramovich, Graber and Vistoli \cite{abramovich06:_gromov_witten_delig_mumfor} constructed an algebraic counterpart of the theory of Chen--Ruan: the orbifold Chow ring and Gromov--Witten theory on a smooth complex Deligne--Mumford stack. Afterwards Borisov, Chen and Smith \cite{borisov05:_chow_delig_mumfor} introduced the concept of stacky fan to construct a toric Deligne--Mumford stack efficiently and compute the orbifold Chow ring of the toric Deligne--Mumford stack in terms of the stacky fan. A stacky fan can be used to study Gromov--Witten theory and mirror symmetry. (See Iritani \cite{iritani07:_real_i} for example.)

The aim of this paper is to introduce a stacky polytope as a counterpart of a stacky fan and to establish a relation between stacky polytopes and stacky fans, applying the stack description developed by Lerman and Malkin. Starting with a stacky polytope, we discuss a construction of symplectic Deligne--Mumford stacks. The main theorem says that the Deligne--Mumford stack associated to a stacky polytope $\vec{\Delta}$ is equivalent to the Deligne--Mumford stack associated to a stacky fan $\vec{\Sigma}$ if $\vec{\Sigma}$ corresponds to $\vec{\Delta}$ (Theorem \ref{thm:equivalence_between_two_stacks}). We do not discuss the orbifold cohomology in this paper, but the terminology of stacky polytopes could be useful when we compute orbifold cohomology in a similar way to Borisov--Chen--Smith. (See also Chen--Hu \cite{chen06:_chen_ruan}.)

This paper is organised as follows. In section 2, we review briefly the theory of stacks and related geometric concepts. In section 3, we discuss symplectic quotients in terms of stacks, using the notion developed by Lerman and Malkin. This construction is well-known for smooth quotients. After that we introduce stacky polytopes and construct the compact symplectic Deligne--Mumford stack associated to a stacky polytope. Then we discuss stacky polytopes which generate torus quotients. For example we describe explicitly a stacky polytope which generates a weighted projective space. In section 4, we briefly review stacky fans and discuss a relation between stacky polytopes and stacky fans.

\section{Symplectic Deligne--Mumford stacks}

In this section, we review briefly the theory of Deligne--Mumford stacks and their differential or symplectic geometry. There are several expository articles of the theory of differentiable stacks: Behrend--Xu \cite{behrend08:_differ_stack_gerbes}, Heinloth \cite{heinloth05:_notes} and Metzler \cite{metzler03:_topol_smoot_stack}. As far as possible we follow the notation of Behrend--Xu.

\subsection{The category of stacks}
\label{subsec:stacks}

To deal with stacks, we first have to fix a category equipped with a Grothendieck topology. We will only use the category of smooth manifolds and smooth maps. This category is denoted by $\Diff$ and equipped with the Grothendieck topology defined as follows. For a smooth manifold $U \in \Diff$ a family $\{f_i: U_i \to U\}_i$ of smooth maps to $U$ is said to be a \emph{covering family} (or just a \emph{covering}) if each $f_i$ is an \etale{} map (i.e.\! a local diffeomorphism) and $\bigcup_i f_i(U_i) = U$. Then the collections of coverings define a Grothendieck topology on $\Diff$ \cite[Section 2]{metzler03:_topol_smoot_stack}.

A \emph{category fibred in groupoids} (over $\Diff$) is a category $\X$ equipped with a functor $F_\X: \X\to\Diff$ which satisfies the following conditions.
\begin{enumerate}[(F1)]
 \item For each smooth map $f: V \to U$ and an object $x \in \X$ with $F_\X(x) = U$, there is an arrow $a: y\to x$ in $\X$ such that $F_\X(a)=f$. The object $y$ is called the \emph{pullback} of $x$ by $f$ and denoted by $f^*x$ or $x|_V$.
 \item If we have two arrows $a: y\to x$ and $b: z\to x$ in $\X$, then for any smooth map $f: F_\X(y)\to F_\X(z)$ with $F_\X(b) \circ f=F_\X(a)$ there is a unique arrow $c: y\to z$ such that $F_\X(c) = f$ and $bc=a$.
\end{enumerate}
By the definition, a category fibred in groupoids $\X$ has the following properties.
\begin{itemize}
 \item An object $x \in \X$ is said to \emph{lie over} $U \in \Diff$ if $F_\X(x) = U$. An arrow $a:y \to x$ is said to \emph{lie over} $U$ if $F_\X(a) = \id_U$. The subcategory consisting of objects and arrows lying over $U$ is denoted by $\X(U)$ and called the \emph{fibre} of $\X$ over $U$. The subcategory $\X(U)$ is a groupoid: a category whose arrows are all invertible.
 \item Because of the condition (F2), the pullback $y$ in the condition (F1) is almost unique in the following sense: for two pullbacks $y$ and $z$ of $x$ by $f$ there is a unique invertible arrow $y \to z$ in $\X(U)$.
 \item \squashup 
       For an arrow $a:y \to x$ lying over $U$ and a smooth map $f:V \to U$, the condition (F2) guarantees that there is a unique arrow $y|_V \to x|_V$. This arrow is called the \emph{pullback} of $a$ by $f$ and denoted by $f^*a$ or $a|V$.
\end{itemize}

\begin{remark}
 In this paper, for any groupoid the source map and the target map are denoted by $\src$ and $\tgt$ respectively. Namely for an arrow $a:y \to x$, we have $\src(a) = y$ and $\tgt(a) = x$.
\end{remark}

A category fibred in groupoids $\X$ is said to be a \emph{stack} (over $\Diff$) if $\X$ satisfies the glueing conditions \cite[Definition 2.4]{behrend08:_differ_stack_gerbes}. (We never use the conditions explicitly.)

The stacks which we mainly deal with in this paper are quotient stacks. If a Lie group $G$ acts on a manifold $M$, we can define the \emph{quotient stack} $[M/G]$ as follows. An object of the category $[M/G]$ is a pair of a principal $G$-bundle $\pi: P \to U$ and a $G$-equivariant map $\epsilon: P \to M$. Here the latter means $\epsilon(p \cdot g) = g^{-1} \cdot \epsilon(p)$ for any $g \in G$ and $p \in P$. The object is written as a diagram $U \stackrel{\pi}{\leftarrow} P \stackrel{\epsilon}{\rightarrow} M$ and often abbreviated to $P$ if no confusion arises. An arrow from $V \stackrel{\pi'}{\leftarrow} Q \stackrel{\epsilon'}{\rightarrow} M$ to $U \stackrel{\pi}{\leftarrow} P \stackrel{\epsilon}{\rightarrow} M$ is a pair of smooth maps $(f,\tilde{f})$ such that $\tilde{f}$ is $G$-equivariant and the following diagram commutes.
\[
 \xymatrix@R=4pt@C=60pt{
 U & P \ar[l]_{\pi} \ar[dr]^{\epsilon} & \\
 & & M. \\
 V \ar[uu]^{f} & Q \ar[l]^{\pi'} \ar[uu]^{\tilde{f}} \ar[ur]_{\epsilon'} & 
}
\]
A functor $F: [M/G]\to\Diff$ is defined by $F(U \leftarrow P \rightarrow M) = U$ and $F(f,\tilde{f}) = f$. 

Pullbacks are given as follows. For a smooth map $f: V \to U$ and an object $U \stackrel{\pi}{\leftarrow} P \stackrel{\epsilon}{\rightarrow} M$ lying over $U$, the pullback $f^*P \in [M/G](V)$ is $V \stackrel{\pi'}{\leftarrow} f^*P \stackrel{\epsilon'}{\rightarrow} M$ in the following commutative diagram.
\[
\xymatrix@R=4pt@C=60pt{
 U & P \ar[l]_{\pi} \ar[dr]^{\epsilon} & \\
 & & M. \\
 V \ar[uu]^{f} & f^*P \ar[l]^{\pi'} \ar[uu]^{\widetilde{f}} \ar[ur]_{\epsilon'}
 & 
 }
\]
Here $f^*P = V \times_{U} P$ is the pullback of the $G$-bundle $P$ to $U$ and $\epsilon'(v,p) = \epsilon(p)$. The universality of the (ordinary) pullback guarantees the condition (F2). Therefore $[M/G]$ is a category fibred in groupoids. Moreover we can see that $[M/G]$ is a stack.

\begin{dfn}
 For stacks $F_\X: \X\to\Diff$ and $F_\Y: \Y\to\Diff$, a \emph{morphism of stacks} is a functor $\phi: \X \to \Y$ satisfying $F_\Y \circ \phi = F_\X$.
 The class of morphisms of stacks from $\X$ to $\Y$ is denoted by $\Mor(\X,\Y)$.
\end{dfn}
\begin{dfn}
 Let $\phi: \X \to \Y$ and $\phi': \X \to \Y$ be two morphisms of stacks. A \emph{map of morphisms} is a natural isomorphism $\alpha:\phi \Rightarrow \phi'$ of functors.
\end{dfn}

The category of stacks (over $\Diff$) forms a $2$-category. Namely for any two stacks $\X$ and $\Y$, the class $\Mor(\X,\Y)$ forms a category again. The class of arrows consists of maps of morphisms. 

In a $2$-category, two morphisms are identified if there is an arrow between them in $\Mor(\X,\Y)$. Moreover two stacks are regarded as the same stack if there is an equivalence between them instead of an isomorphism. 

\begin{dfn}
 Two stacks $\X$ and $\Y$ are said to be \emph{equivalent} if there are morphisms of stacks $\phi:\X \to \Y$ and $\psi:\Y \to \X$ such that there are maps of morphisms $\id_\X \Rightarrow \psi\circ\phi$ and $\id_\Y \Rightarrow \phi\circ\psi$.
\end{dfn}

\begin{remark}
 Some authors use the terminology ``isomorphic'' instead of ``equivalent''. In this paper we follow the usual terminology of the theory of $2$-categories.
\end{remark}


A manifold $M$ can naturally be considered as a stack $\X_M$ as follows: the class of objects consists of smooth maps whose target is $M$ and an arrow from $f:U \to M$ and $g:V \to M$ is a smooth map $a: U \to V$ such that $g \circ a = f$. The fibre of $\X_M$ over $U$ is given by $\mathcal{C}^\infty(U,M)$. Note that we consider the set $\mathcal{C}^\infty(U,M)$ as a discrete category. 
Thanks to the Yoneda lemma, the category $\Diff$ is fully embedded into the category of stacks. Therefore we identify $M$ with $\X_M$.

\begin{dfn}
 A stack is said to be \emph{representable} if there is a manifold which is equivalent to the stack.  
\end{dfn}

\subsection{Deligne--Mumford stacks}
\label{subsec:DM_stacks}

Note that we can always take a ($2$-)fibred product in the category of stacks. 


\begin{dfn}
 A morphism of stacks $p$ from a manifold $X$ to a stack $\X$ is called an \emph{atlas} (resp. \emph{\etale{} atlas}) of $\X$ if for any morphism of stacks from a manifold $Y$ to $\X$
 \begin{itemize}
  \item the fibred product $X \times_{\X} Y$ is representable, and
  \item the projection $X \times_{\X} Y \to Y$ is a surjective submersion (resp. surjective locally diffeomorphism).
 \end{itemize}
\end{dfn}

If $p: X_0 \to \X$ is an atlas of a stack $\X$, then the fibred product $X_0 \times_{\X} X_0$ is equivalent to a manifold $X_1$ and there are two surjective submersions $\src,\tgt: X_1 \to X_0$ as projections:
\[
 \xymatrix@C=50pt{
 X_1 \ar[r]^{\tgt} \ar[d]_{\src} & X_0 \ar[d]^{p} \\
 X_0 \ar[r]_{p} & \X.
 }
\]
Then $X_1 \rightrightarrows X_0$ form a Lie groupoid. 
The Lie groupoid $X_1 \rightrightarrows X_0$ is said to be \emph{associated to the atlas} $p: X_0 \to \X$. Note that the above diagram is $2$-commutative, i.e.\! there is a map of morphisms from $p \circ \src$ to $p \circ \tgt$. 

\begin{dfn}
 Let $\X$ be a stack having an atlas $p: X_0 \to \X$. We say that the stack $\X$ is \emph{separated} if the map 
 \[
 X_1 \to X_0 \times X_0;\ a \mapsto (\src(a),\tgt(a))
 \]
 is proper. This definition is independent of the choice of the atlas \cite[Section 2.4]{behrend08:_differ_stack_gerbes}.
\end{dfn}

\begin{dfn}
 A stack is said to be \emph{differentiable} (resp. \emph{Deligne--Mumford}) if the stack is separated and admits an atlas (resp. \etale{} atlas). 
\end{dfn}


The \emph{underlying space} of the differentiable stack $\X$ is the topological space $|\X| = X_0/\!\!\sim$, where $x \sim y$ for $x, y \in X_0$ if there is $a \in X_1$ such that $\src(a)=x$ and $\tgt(a)=y$. The topology of the underlying space is well-defined. A differentiable stack is said to be \emph{compact} if its underlying space is compact. 

Each quotient stack has an atlas. For a $G$-action on $M$, the \emph{natural projection} is a morphism $p:  M \to [M/G]$ defined by $p(f: U \to M) = \bigl( U \stackrel{\pi}{\leftarrow} G \times U \stackrel{\epsilon}{\rightarrow} M \bigr)$. Here $\pi$ is the projection, $\epsilon(g,u) = g \cdot f(u)$ and the right $G$-action on $G \times U$ is defined by $(g,u) \cdot h = (h^{-1}g,u)$. For every arrow $a: f' \to f$ in $M$, $p(a)$ is naturally defined. The groupoid associated to the atlas $p$ is the action groupoid $G \times M \rightrightarrows M$. Therefore if the $G$-action on $M$ is proper, then the quotient stack is differentiable. The underlying space of $[M/G]$ is the quotient topological space $M/G$.

\begin{prop}
 Let $M$ be a manifold equipped with a smooth action of a Lie group $G$. If the action is proper and locally free, then there is an \etale{} atlas of the quotient stack $[M/G]$ i.e.\! the quotient stack $[M/G]$ is Deligne--Mumford.
\end{prop}
This theorem can be showed by using Theorem 1 in Crainic--Moerdijk \cite{crainic01:_foliat} (cf. Lerman--Malkin {\cite[Theorem 2.4]{0908.0903}}). 

\subsection{Differential forms over a differentiable stack}
\label{subsec:forms_on_stacks}

A (global) differential form over a differentiable stack is defined as a global section of the sheaf of differential forms.  A presheaf is usually defined as a contravariant functor and a sheaf can be defined over a category equipped with a Grothendieck topology. The Grothendieck topology on a stack $\X$ can be induced by the Grothendieck topology of $\Diff$. Moreover for a sheaf over a differentiable stack $\X$, we can define the set of global sections. Details can be found in Behrend--Xu \cite{behrend08:_differ_stack_gerbes} and Metzler \cite{metzler03:_topol_smoot_stack}.


We define the sheaf $\Omega^k_\X$ of smooth $k$-forms on a differentiable stack $\X$ as follows. For an object $x \in \X$ lying over $U$, put $\Omega^k_\X(x) = \Omega^k(U)$. For an arrow $a:y \to x$ in $\X$ with $F_\X(a) = f:V \to U$, we assign to $\Omega^k_\X(x) \to \Omega^k_\X(y)$ the pullback $f^*:\Omega^k(U) \to \Omega^k(V)$. If $X_1 \rightrightarrows X_0$ is the groupoid associated to an atlas $p: X_0 \to \X$, then the set of global sections of $\Omega_\X^k$ is given by
\[
 \Omega^k(\X) = 
 \bigl\{ \eta \in \Omega^k(X_0) \!\ \big|\!\ \src^*\eta = \tgt^*\eta \bigr\}.
\]
An element of $\Omega^k(\X)$ is called a (global) $k$-form on $\X$. 

\begin{examples}
 Suppose that a manifold $M$ is equipped with a smooth proper action of a Lie group $G$. Then the set of global $k$-forms on the quotient stack $[M/G]$ is given by
 \[
 \Omega^k([M/G]) = \bigl\{
 \eta \in \Omega^k(M) \big| \text{$\eta$ is $G$-invariant and $\iota(\xi_M)\eta = 0$ for any $\xi \in \g$}\bigr\}.
 \]
 Here $\iota$ is the interior multiplication, $\g$ is the Lie algebra of the Lie group $G$ and $\xi_M$ is the infinitesimal action of $\xi$, i.e.\! $\xi_M(x) = \left.\frac{d}{d\lambda}\right|_{\lambda=0} \exp(\lambda\xi) \cdot x$.
\end{examples}

\subsection{Vector fields and symplectic forms over a Deligne--Mumford stack}
\label{subsec:vector_fields_on_stacks}

Let $p:X_0 \to \X$ be an atlas and $X_1 \rightrightarrows X_0$ the groupoids associated to the atlas. We define the groupoid of vector fields $\Vect(X_1 \rightrightarrows X_0)$ over $X_1 \rightrightarrows X_0$ as follows.

Note that we have the Lie groupoid $TX_1 \rightrightarrows TX_0$ whose structure maps are defined by derivatives of the structure maps of $X_1 \rightrightarrows X_0$. The projections define a smooth functor $\pi$ from $TX_1 \rightrightarrows TX_0$ to $X_1 \rightrightarrows X_0$. The set of objects $\Vect(X_1 \rightrightarrows X_0)$ consists of smooth functors $v$ from $X_1 \rightrightarrows X_0$ to $TX_1 \rightrightarrows TX_0$ satisfying that the composition $\pi \circ v$ is equal to the identity functor. An arrow $v \to v'$ is a natural isomorphism $\alpha:v \to v'$ satisfying that the horizontal composition $\id_{\pi}*\alpha$ is equal to the identity transformation of the identity functor for $X_1 \rightrightarrows X_0$.

\begin{thm}
 [Hepworth {\cite[Theorem 3.13]{hepworth09:_vector}}] The groupoid $\Vect(X_1 \rightrightarrows X_0)$ is independent of the choice of the atlas up to category equivalences.
\end{thm}

The vector space of the equivalence classes for the groupoid $\Vect(X_1 \rightrightarrows X_0)$ is denoted by $\Vect(\X)$ and called the space of vector fields over $\X$. 

\begin{thm}
 [Lerman--Malkin {\cite[Proposition 2.9]{0908.0903}}]
 Let $\X$ be a stack with an atlas $p:X_0 \to \X$ and $X_1 \rightrightarrows X$ the groupoid associated to the atlas. If $\X$ is Deligne--Mumford, then $\Vect(\X)$ is given by the following quotient vector space:
 \[
 \mathcal{V}\big/\{(v_1,v_0) \in \mathcal{V}\!\ |\!\ v_1 \in \ker(d\src) + \ker(d\tgt)\}.
 \]
 Here $\mathcal{V}$ is the vector space consisting of pairs $(v_1,v_0) \in \Vect(X_1) \times \Vect(X_0)$ of (ordinary) vector fields satisfying $d(\src) \circ v_1 = v_0 \circ \src$ and $d(\tgt) \circ v_1 = v_0 \circ \tgt$.
\end{thm}

Following Lerman--Malkin \cite{0908.0903}, we introduce a symplectic form on a Deligne--Mumford stack as follows. Let $\X$ be a Deligne--Mumford stack with an atlas $p: X_0 \to \X$. A vector field over $\X$ can be represented by an equivalence class $v = [v_1,v_0]$ of pair of vector fields as above. The interior multiplication $\iota(v)$ is defined by
\[
 \iota(v):  \Omega^2(\X) \to \Omega^1(\X);\ \omega \mapsto \iota(v_0)\omega.
\]
Here $\omega$ is represented as a $2$-form on $X_0$. 

\begin{remark}
 Lerman and Malkin describe a differential form over $\X$ as a pair of differential forms $(\omega_1,\omega_0)$ such that $\omega_1 = \src^*\omega_0 = \tgt^*\omega_0$ \cite{0908.0903}. We omit $\omega_1$, because it is redundant.
\end{remark}

A $2$-form $\omega$ on $\X$ is said to be \emph{nondegenerate} if the map 
\[
 \Vect(\X) \to \Omega^1(\X);\ v \mapsto \iota(v)\omega
\]
is a linear isomorphism. This is equivalent to the condition that $\ker \omega = \A$. Here $\A$ is the Lie algebroid of the groupoid associated to the atlas $p: X_0 \to \X$. In other words $\A$ is the pullback bundle of $\ker(d\src)$ by the unit map $X_0 \to X_1$. For a Deligne--Mumford stack we can regard the bundle $\A \to X_0$ as a subbundle of the tangent bundle $TX_0$ via $d\tgt$  \cite[Theorem 2.4]{0908.0903}. A nondegenerate closed $2$-form on $\X$ is called a \emph{symplectic form} on $\X$.

\section{The Deligne--Mumford stack associated to a stacky polytope}
\label{section:DM_stack_associated_to_stacky_polytope}

Borisov et al. define stacky fans and construct the Deligne--Mumford stack $\X_{\vec\Sigma}$ associated to a stacky fan $\vec{\Sigma}$ \cite{borisov05:_chow_delig_mumfor}. Motivated by Borisov et al., we define stacky polytopes which are symplectic counterparts and construct the symplectic Deligne--Mumford stack $\X_{\vec\Delta}$ associated to a stacky polytope $\vec{\Delta}$.

\subsection{A symplectic quotient as a stack} \label{sec:symp_toric_quot}

In this subsection we discuss a construction of symplectic quotients in terms of stacks. This is well-known as construction of symplectic orbifolds. 

\begin{remark}
 Lerman and Malkin construct symplectic Deligne--Mumford stacks in a different way to our construction \cite{0908.0903}. But we stick with the standard construction because it is useful when we find a stacky polytope for a given symplectic quotient in Subsection \ref{subsection:stacky_polytope_of_torus_quotients}.
\end{remark}

The $d$-dimensional torus $\T^d = \R^d/\Z^d$ linearly acts on $\C^d$:
\[
 \T^d \acts \C^d; \quad
 [\theta_1,\dots,\theta_d] \cdot (z_1,\dots,z_d) = 
 (e^{-2\pi\i \theta_1}z_1,\dots,e^{-2\pi\i \theta_d}z_d). 
\]
Here $\i=\sqrt{-1}$. The action is Hamiltonian with respect to the standard symplectic structure on $\C^d$. A moment map of the action is given by 
\begin{equation}
 \label{eq:mu_zero}
 \mu_0 :  \C^d \to (\R^d)^\dual; \quad
 z \mapsto \pi\sum_{\alpha=1}^d |z_\alpha|^2\vec{e}^\alpha. 
\end{equation}
Here ``$\dual$'' means taking the dual space and $\vec{e}^1,\dots,\vec{e}^d$ are the dual basis of the standard basis $\vec{e}_1,\dots,\vec{e}_d$ of $\R^d$.

Let $G$ be a compact Lie group whose adjoint representation is trivial. It is easy to see that the identity component $G_0$ of $G$ is a compact torus. Given a homomorphism $\rho: G \to \T^d$, we define the smooth action of $G$ on $\C^d$ through the homomorphism. Let $\rho^\dual: (\R^d)^\dual \to \g^\dual$ be the induced linear map, where $\g$ is the Lie algebra of $G$. If we put $w^\alpha = \rho^\dual(\vec{e}^\alpha)$, then a moment map of the $G$-action is given by 
\begin{equation}
 \label{eq:moment_map_of_G}
 \mu :  \C^d \to \g^\dual; \quad
 z \mapsto \pi\sum_{\alpha=1}^d |z_\alpha|^2 w^\alpha. 
\end{equation}
The elements $w^1,\dots,w^d$ are called \emph{weights}.

Since the moment map is $G$-equivariant, for $\tau \in \g^\dual$ the level set $\mu^{-1}(\tau)$ is closed under the continuous action of $G$. If $\tau$ is a regular value of $\mu$, the level set $\mu^{-1}(\tau)$ is a smooth manifold. Moreover it follows from the next lemma that the $G$-action on $\mu^{-1}(\tau)$ is locally free.

\begin{lem}
 \label{lem:reg-val-and-local-freeness}
 For $\tau \in \g^\dual$, $\tau$ is a regular value of $\mu$ if and only if for $z \in \mu^{-1}(\tau)$ and $\xi \in \g$ the identity $\xi_{\C^d}(z)=0$ implies $\xi=0$.
\end{lem}

\begin{proof}
A covector $\tau \in \g^\dual$ is a regular value of $\mu$ if and only if for any $z \in \mu^{-1}(\tau)$ the derivative $d\mu(z): \C^d \to \g^\dual$ is surjective. The map $d\mu(z)$ is surjective if and only if $\pair{d\mu(z),\xi}|_z=0$ implies $\xi=0$  for $\xi \in \g$. Since $\pair{d\mu(z),\xi}|_z = \iota(\xi_{\C^d})\omega_0|_z$, we can conclude the lemma.
\end{proof}

If $\tau \in \g^\dual$ is a regular value of $\mu$ and $\lm$ is not empty, then the $G$-action on the level manifold $\mu^{-1}(\tau)$ is proper and locally free. Therefore the quotient stack $\C^d\sq{\tau}G = [\mu^{-1}(\tau)/G]$ is a Deligne--Mumford stack. 

The restriction $\omega$ of the standard symplectic form $\omega_0$ on $\C^d$ to the level manifold $\mu^{-1}(\tau)$ is $G$-invariant. Since $\ker\omega$ coincides with the Lie algebroid of the action groupoid $G \times \lm \rightrightarrows \lm$ \cite[Proposition 5.40]{mcduff98:_introd}, $\omega$ is a symplectic form on the stack $\C^d \sq{\tau} G$. We call the quotient stack $\C^d \sq{\tau} G$ the \emph{symplectic quotient}.

The moment map $\mu: \C^d\to\g^\dual$ is often assumed to be proper because this assumption guarantees that the symplectic quotient is compact. 
\begin{lem}
 [Guillemin--Ginzburg--Karshon {\cite[Proposition 4.14]{guillemin02:_momen_hamil}}]
 \label{lem:properness} 
 The moment map $\mu: \C^d \to \g^\dual$ is proper if and only if there exists a covector $\tau = \sum_{\alpha=1}^d s_\alpha w^\alpha$ with $s_\alpha \geq 0\ (\alpha=1,\dots,d)$ such that 
$\{ s \in (\R^d)^\dual | \pair{s,\vec{e}_\alpha} \geq 0,\ \rho^\dual(s)=\tau \} $ is compact.
\end{lem}

This subsection can be summarised as the following theorem.
\begin{thm}
 \label{thm:symp_quot_is_symp_DM_stack}
 Let $(G,\rho,\tau)$ consist of 
 \begin{itemize}
  \item \squashup
	a compact Lie group $G$ whose adjoint representation is trivial,
  \item \squashup
	a homomorphism $\rho: G \to \T^d$ of Lie groups, and
  \item \squashup
	a regular value $\tau \in \g^\dual$ of the moment map \eqref{eq:moment_map_of_G}.
 \end{itemize}
 We assume that the triple satisfies the following conditions.
 \begin{enumerate}
  \item The level set $\lm$ is nonempty.
  \item The moment map $\mu: \C^d\to\g^\dual$ is proper. 
 \end{enumerate}
Then the quotient stack $\C^d\sq{\tau}G = [\lm/G]$ is a compact symplectic Deligne--Mumford stack. 
\end{thm}


\subsection{A stacky polytope}
\label{subsection:stacky_polytope}

In this subsection we define a stacky polytope and construct a symplectic Deligne--Mumford stack from a stacky polytope.

\begin{dfn}
 \label{dfn:stacky_polytopes}
 Consider a triple $(N,\Delta,\beta)$ of 
 \begin{itemize}
  \item \squashup
	a finitely generated $\Z$-module $N$ of rank $r$,
  \item \squashup
	a polytope $\Delta$ with $d$ facets $F_1,\dots,F_d$ in $\t^\dual=N^\dual\otimes_\Z\R$, and
  \item \squashup
	a homomorphism of $\Z$-modules $\beta :  \Z^d \to N$.
 \end{itemize}
 We call the triple $(N,\Delta,\beta)$ a \emph{stacky polytope} if the following conditions are satisfied.
 \begin{enumerate}
  \item \squashup
	The polytope $\Delta$ is simple, i.e.\ every vertex is contained in exactly $r$ facets.
  \item \squashup
	Let $e_1,\dots, e_d$ be the standard $\Z$-basis of $\Z^d$. The natural map $N \to \t = N \otimes_\Z \R$ is denoted by $n \mapsto \overline{n}$. Let $\Lambda$ be the image of $N$ via the natural map. Then $\overline{\beta(e_1)},\dots, \overline{\beta(e_d)} \in \Lambda$ are vectors perpendicular to the facets $F_1,\dots, F_d$ in inward-pointing way, respectively.
  \item \squashup
	The cokernel of the homomorphism $\beta: \Z^d \to N$ is finite. 
 \end{enumerate}
\end{dfn}

The second condition implies that the polytope $\Delta$ is rational and can be described as
\begin{equation}
 \label{eqn:polytope-as-intersection-of-hyperplanes}
 \Delta = \left\{ \eta \in \t^\dual \!\ \middle| \!\ \pair{\eta,\overline{\beta(e_\alpha)}} \geq -c_\alpha \right\} 
\end{equation}
for some $c_\alpha \in \R\ (\alpha = 1,\dots,d)$.

We construct the symplectic Deligne--Mumford stack $\X_{\vec\Delta}$ associated to a stacky polytope $\vec{\Delta}=(N,\Delta,\beta)$ as follows. First we construct a homomorphism of $\Z$-modules $\beta^\DG:(\Z^d)^\dual \to \DG(\beta)$ in the same way to Borisov et al. \cite{borisov05:_chow_delig_mumfor}. 

\begin{remark}
 In Borisov--Chen--Smith \cite{borisov05:_chow_delig_mumfor}, the homomorphism $(\Z^d)^\dual \to \DG(\beta)$ is denoted by $\beta^\dual$ instead of $\beta^\DG$. In this paper ``${}^\dual$'' always means ``dual'' in the usual sense. Therefore $\beta^\dual$ is the induced homomorphism $N^\dual \to (\Z^d)^\dual$ by $\beta: \Z^d \to N$, where $(\Z^d)^\dual = \Hom_\Z(\Z^d,\Z)$ and $N^\dual = \Hom_\Z(N,\Z)$.
\end{remark}

Take a projective resolution of $N$, i.e.\! an exact sequence of $\Z$-modules
\[
 \xymatrix{
 \cdots \ar[r]^{\rd_{\vec{F}}} &
 F_2 \ar[r]^{\rd_{\vec{F}}} &
 F_1 \ar[r]^{\rd_{\vec{F}}} &
 F_0 \ar[r] &
 N \ar[r] &
 0
 }
\]
with all the $F_i$'s free over $\Z$. We also take a projective resolution of $\Z^d$
\[
 \xymatrix{
 \cdots \ar[r]^{\rd_{\vec{E}}} &
 E_2 \ar[r]^{\rd_{\vec{E}}} &
 E_1 \ar[r]^{\rd_{\vec{E}}} &
 E_0 \ar[r] &
 \Z^d \ar[r] &
 0.
 }
\]
Then the homomorphism $\beta: \Z^d \to N$ lifts to a chain map $\beta: \vec{E}\to\vec{F}$ \cite[Theorem 2.2.6]{weibel94}, where $\vec{F}=\{F_i, \rd_{\vec{F}}\}_{i \geq 0}$ and $\vec{E} = \{E_i, \rd_{\vec{E}}\}_{i \geq 0}$. The mapping cone $\Cone(\beta)$ is defined as a chain complex $\{E_{i-1} \oplus F_{i}, \rd_{\Cone(\beta)} \}_{i \geq 0}$ whose boundary operator $\rd_{\Cone(\beta)}$ is given by
\[
 \rd_{\Cone(\beta)} :  E_{i-1} \oplus F_i \to E_{i-2} \oplus F_{i-1};\
 (e,f) \mapsto \bigl(-\rd_{\vec{E}}(e),\rd_{\vec{F}}(f)-\beta(e)\bigr).
\]
Then the mapping cone naturally fits into a short exact sequence of chain complexes:
\[
 \xymatrix{
 0 \ar[r] & \vec{F} \ar[r] & \Cone(\beta) \ar[r] & \vec{E}[1] \ar[r] & 0,
 }
\]
where $\vec{E}[1]$ is the chain complex whose $i$-th term is $E_{i+1}$. The dual sequence
\[
 \xymatrix{ 
 0 \ar[r] & 
 \vec{E}[1]^\dual \ar[r] & 
 \Cone(\beta)^\dual \ar[r] & 
 \vec{F}^\dual \ar[r] & 
 0
 }
\]
is a short exact sequence of cochain complexes. It induces a long exact sequence that contains the following part:
\begin{equation}
 \label{eqn:part_of_the_long_exact_sequence_containing_cohomology_of_Cone}
 \xymatrix@C=16pt{
  0 \ar[r] &
  N^\dual \ar[r]^{\beta^\dual} &
  (\Z^d)^\dual \ar[r] &
  \H^1\bigl(\Cone(\beta)^\dual\bigr) \ar[r] &
  \Ext_\Z^1(N,\Z) \ar[r] &
  0.
 }
\end{equation}
Note that the finiteness of $\coker(\beta)$ makes $\beta^\dual$ injective. Denote $\H^1\bigl(\Cone(\beta)^\dual\bigr)$ by $\DG(\beta)$ and define $\beta^\DG: (\Z^d)^\dual\to\DG(\beta)$ as the second homomorphism in the above sequence. Both $\DG(\beta)$ and $\beta^\DG$ are well-defined up to natural isomorphism \cite{borisov05:_chow_delig_mumfor}.

To construct the stack $\X_{\vec\Delta}$ as a symplectic quotient, we give a triple $(G,\rho,\tau)$ satisfying the assumptions of Theorem \ref{thm:symp_quot_is_symp_DM_stack}.

Since $N^\dual$ is naturally isomorphic to $\Lambda^\dual$, the following exact sequence forms part of the exact sequence \eqref{eqn:part_of_the_long_exact_sequence_containing_cohomology_of_Cone}:
\begin{equation}
 \label{eqn:exact_sequence_including_DG}
 \xymatrix@C=40pt{
 0 \ar[r] &
 \Lambda^\dual \ar[r]^(0.45){\beta^\dual} &
 (\Z^d)^\dual \ar[r]^{\beta^\DG} &
 \DG(\beta).
}
\end{equation}
Since the torus $\T$ is injective as a $\Z$-module, the functor $\Hom_\Z(-,\T)$ is exact. Applying the functor to the sequence \eqref{eqn:exact_sequence_including_DG}, we obtain the exact sequence of Lie groups:
\begin{equation}
 \label{eqn:exact_seq_which_we_start_with}
  \xymatrix@C=40pt{
  G \ar[r]^{\rho} & \T^d \ar[r]^{\sigma} & T \ar[r] & \{\1\},
  }
\end{equation}
where $G=\Hom_\Z\bigl(\DG(\beta),\T\bigr)$ and $T=\Hom_\Z\bigl({\Lambda}^\dual,\T\bigr)$. Note that we can identify $\Hom_\Z((\Z^d)^\dual,\T) = \Z^d \otimes \T$ with $\T^d = \R^d/\Z^d$. The homomorphisms $\rho$ and $\sigma$ are induced by $\beta^\DG$ and $\beta^\dual$ respectively. Since $\DG(\beta)$ is an finitely generated $\Z$-module, $G$ is a compact abelian Lie group. As Section~\ref{sec:symp_toric_quot}, the group $G$ acts on $\C^d$ with a moment map $\mu$ given by \eqref{eq:moment_map_of_G}.

\begin{lem}
 Set $\tau = \sum_{\alpha=1}^d c_\alpha w^\alpha$, where $c_\alpha$ is the constant appearing Equation (\ref{eqn:polytope-as-intersection-of-hyperplanes}) and $w^\alpha = \rho^\dual(e^\alpha)$. Then $\tau$ is a regular value of the moment map $\mu$.
\end{lem}

The above lemma is proved by using Lemma \ref{lem:reg-val-and-local-freeness}. (See also the proof of Proposition 5.15 in Guillemin--Ginzburg--Karshon \cite{guillemin02:_momen_hamil}.)

\begin{lem}
 \label{lem:compactness_of_stack_associated_to_stacky_polytope}
 The triple $(G,\rho,\tau)$ satisfies the assumptions of Theorem \ref{thm:symp_quot_is_symp_DM_stack}.
\end{lem}
\begin{proof}

By Lemma \ref{lem:properness}, 
it suffices to show that 
\begin{enumerate}
 \item $\tau \in \bigl\{ \sum_{\alpha=1}^d s_\alpha w^\alpha \!\ \big|\!\ s_\alpha \geq 0 \bigr\}$, and 
 \item the set $\Delta_\tau = \bigl\{	s \in (\R^d)^\dual \ \big|\!\ \pair{s,\vec{e}_\alpha} \geq 0,\ \rho^\dual(s)=\tau \bigr\}$ is compact.
\end{enumerate}
The first condition is equivalent that $\Delta_\tau \ne \varnothing$. Putting $\tau' = \sum_{\alpha=1}^d c_\alpha \vec{e}^\alpha$, we have 
\begin{align*}
 \sigma^\dual(\Delta) + \tau'
 &= \bigl\{
 \sigma^\dual(\eta)+\tau' \in (\R^d)^\dual \big|
 \pair{\eta,\sigma(\vec{e}_\alpha)} \geq -c_\alpha
 \bigr\} \\
 &= \bigl\{
 s \in (\R^d)^\dual \big|
 \pair{s,\vec{e}_\alpha} \geq 0,\
 \rho^\dual(s) = \tau
 \bigr\} \\
 &= \Delta_\tau.
\end{align*}
Since $\Delta$ is nonempty and compact, so is $\Delta_\tau$.
\end{proof}

Applying Proposition \ref{thm:symp_quot_is_symp_DM_stack} to the triple $(G,\rho,\tau)$, we finally obtain a compact symplectic Deligne--Mumford stack $\C^d\sq{\tau}G$. 

\begin{thm}
 Let $\vec{\Delta}=(N,\Delta,\beta)$ be a stacky polytope. Define $(G,\rho,\tau)$ by a triple consisting of 
 \begin{itemize}
  \item the Lie group $G = \Hom_\Z(\DG(\beta),\T)$,
  \item the homomorphism $\rho:G \to \T^d$ induced by $\beta^{\DG}:(\Z^d)^\dual \to \DG(\beta)$, and
  \item the covector $\tau=\sum_{\alpha=1}^d c_\alpha w^\alpha$.
 \end{itemize}
Here $d$ and $c_\alpha$ are the constants appearing Definition \ref{dfn:stacky_polytopes} and $w^\alpha = \rho^\dual(e^\alpha)$. Then the triple $(G,\rho,\tau)$ satisfies the assumptions of Theorem \ref{thm:symp_quot_is_symp_DM_stack}. 
\end{thm}
By Theorem \ref{thm:symp_quot_is_symp_DM_stack}, we obtain the symplectic quotient $\C^d\sq{\tau}G$ associated to the above triple $(G,\rho,\tau)$. We call the symplectic quotient the \emph{symplectic Deligne--Mumford stack $\X_{\vec\Delta}$ associated to the stacky polytope $\vec{\Delta}$}.

\begin{remark}
For the symplectic Deligne--Mumford stack associated to a stacky polytope, all stabilizer groups are abelian because $\Hom_\Z(\DG(\beta),\T)$ is abelian.
\end{remark}

\begin{remark}
Lerman and Tolman defined labelled polytopes to classify compact symplectic toric orbifolds and established a construction (Lerman--Tolman--Delzant construction) of compact symplectic toric orbifolds from labelled polytopes \cite{LT}. Their construction can also be used to produce symplectic Deligne--Mumford stacks. Ignoring a Hamiltonian structure, our construction is slightly wider than the Lerman--Tolman--Delzant construction in the following sense: Let $(N,\Delta,\beta)$ be a stacky polytope. Suppose $N$ to be free over $\Z$. Define a label $m_\alpha$ of $\alpha$-th facet $F_\alpha$ by the identity $\beta(e_\alpha) = m_\alpha \nu_\alpha$, where $\nu_\alpha$ is the primitive inward-pointing vector perpendicular to the facet $F_\alpha$. Then the rational convex polytope $\Delta$ together with the labels $m_1, \dots, m_d$ associated to the facets $F_1, \dots, F_d$ gives a labelled polytope. On the other hand every labelled polytope arises in this way.
\end{remark}

\subsection{The stacky polytope of a torus quotient}
\label{subsection:stacky_polytope_of_torus_quotients}


Given a symplectic Deligne--Mumford stack $\X$ (with abelian stabilizer groups) it is natural to ask whether there is a stacky polytope $\vec{\Delta}$ such that $\X_{\vec{\Delta}}$ is equivalent to $\X$. In this section we give a partial solution to this question: If $G$ is a torus and $\X$ is the symplectic Deligne--Mumford stack $\X$ associated to a triple $(G,\rho,\tau)$ satisfying the assumptions in Theorem \ref{thm:symp_quot_is_symp_DM_stack}, then we can find a stacky polytope $\vec{\Delta}$ in such a way that the associated stack $\X_{\vec\Delta}$ is equivalent to $\X$.

Define $\Z_G $ by $\ker(\exp:\Lie(G) \to G)$. Then the homomorphism $\rho$ induces a monomorphism of $\Z$-modules $\dot\rho: \Z_G \to \Z^d$ and we have a short exact sequence
\begin{equation}
 \label{eqn:projective_resolution_of_N}
 \xymatrix@C=30pt{
  0 \ar[r] &
  \Z_G \ar[r]^<(0.35){\dot\rho} &
  \Z^d \ar[r]^<(0.35){\beta} &
  N \ar[r] &
  0,
  }
\end{equation}
where $N$ is the cokernel of $\dot\rho$ and $\beta: \Z^d \to N$ is the natural quotient homomorphism. The image $\Lambda$ of $N$ in the vector space $\t=N\otimes\R$ is a lattice of $\t$ and the torus $T = \t/\Lambda$ can be naturally identified with the quotient torus $\T^d/\rho(G)$. The composition map of $\beta$ and the natural projection $N\to\Lambda$ gives rise to a homomorphism of tori $\sigma: \T^d \to T$. 

We may take a homomorphism $s: T\to\T^d$ satisfying $\sigma \circ s = \id_T$. Consider the following map:
\[
 \bar{\mu}: \lm \to \t^\dual;\ \ z \mapsto 
 \sum_{\alpha=1}^d \bigl(\pi|z_\alpha|^2-c_\alpha \bigr)s^\dual(\vec{e}^\alpha).\]
Here $c_1, \dots, c_d$ are real numbers satisfying $\sum_\alpha c_\alpha w^\alpha = \tau$. The image of the map is a convex polytope
\[
 \Delta = \bigl\{\eta \in \t^\dual \ \big| \pair{\eta,\bar{n}_\alpha} \geq -c_\alpha\ (\alpha=1,\dots,d) \bigr\},
\]
where $n_\alpha = \beta(\vec{e}_\alpha)$ and $\bar{n}_\alpha$ is the image of $n_\alpha$ via the natural projection $N\to\Lambda$. 

\begin{lem}
 The triple $\vec{\Delta} = (N,\Delta,\beta)$ is a stacky polytope.
\end{lem}
\begin{proof}
The conditions in Definition~\ref{dfn:stacky_polytopes} are obviously satisfied except the simplicity of $\Delta$. For $I \subset \{1,\dots,d\}$ we define a subset $\Delta_I$ of $\Delta$ by 
\[
 \Delta_I = \{ \eta \in \Delta \ \big|
 \pair{\eta,\bar{n}_\alpha} = -c_\alpha\ \text{ for } \alpha \in I
 \}.
\]
Each face of $\Delta$ can be described as $\Delta_I$ for some $I$. According to Cieliebak--Salamon \cite[Lemma E.1]{cieliebak06:_wall}, the set $\Delta_I$ is empty or has codimension $|I|$ in $\t^\dual$. Note that $\dim \Delta = \dim \Delta_\varnothing = \dim \t^\dual$. Suppose $\Delta_I$ is a vertex of $\Delta$. If $\Delta_J$ is a nonempty facet containing the vertex $\Delta_I$, then $J \subset I$. Because $|J| = \codim \Delta_J = 1$, $J = \{j\}$ for some $j \in I$. On the other hand, for any $j \in I$, the set $\Delta_{\{j\}}$ is a nonempty facet containing the vertex $\Delta_I$.

The above discussion implies that the number of nonempty facets containing the vertex $\Delta_I$ is equal to $|I|$. Since $|I|=\codim\Delta_I=\dim\t^\dual$, the convex polytope $\Delta$ is simple. 
\end{proof}

\begin{thm}
 The symplectic Deligne--Mumford stack $\X_\vec{\Delta}$ associated to the stacky polytope $\vec{\Delta}$ is equivalent to $\C^d\sq{\tau}G$.
\end{thm}
\begin{proof}
Note that the short exact sequence \eqref{eqn:projective_resolution_of_N} gives a projective resolution of the $\Z$-module $N$. The dual of the mapping cone is given by
\[
 \xymatrix@C=30pt{0 \ar[r] & (\Z^d)^\dual \ar[r]^(0.37){d} & (\Z^d)^\dual\oplus(\Z_G)^\dual \ar[r] & 0 \ar[r] & \cdots.}
\]
Here the differential $d$ is explicitly given by $d\vec{e}^\alpha = (-\vec{e}^\alpha,\rho^\dual(\vec{e}^\alpha))\ (\alpha=1,\dots,d)$. We can identify $\DG(\beta) = \coker(d)$ with $(\Z_G)^\dual$ by the map
\[
 \DG(\beta) \to (\Z_G)^\dual;\ 
 [\vec{e}^\alpha,w] \mapsto \rho^\dual(\vec{e}^\alpha)+w.
\]
Under this identification, the abelian Lie group $\Hom_\Z(\DG(\beta),\T)$ is the torus $G$ and the homomorphism $\Hom_Z(\DG(\beta),\T) \to \T^d$ induced by $\beta^\DG$ in the exact sequence \eqref{eqn:exact_sequence_including_DG} is the same as the homomorphism $\rho: G \to \T^d$. Therefore $\X_\vec{\Delta}$ and $\C^d\sq{\tau}G$ are both defined by the same data $(G,\rho,\tau)$. 
\end{proof}

\begin{examples}
 Let $\vec{w}=(w_1,\dots,w_d)$ be a $d$-tuple of positive integers. The \emph{weighted projective space} $\P(\vec{w})$ of weight $\vec{w}$ is defined as follows. Consider the homomorphism $\rho:  \T \to \T^d$ defined by $\rho([\xi])=[w_1\xi,\dots,w_d\xi]$, where $\xi \in \R = \Lie(\T)$. Then the $\T$-action on $\C^d$ via $\rho$ is given by 
 \[
 \T \acts \C^d;\ 
 [\xi] \cdot (z_1,\dots,z_d) 
 = (e^{-2\pi\i w_1\xi}z_1,\dots,e^{-2\pi\i w_d\xi}z_d).
 \]
 The following map gives a moment map of the action:
 \[
 \mu:  \C^d \to \R;\ 
 (z_1,\dots,z_d) \mapsto \sum_{\alpha=1}^d \pi |z_\alpha|^2  w_\alpha.
 \]
 Here we identify $\Lie(\T)^\dual=\R^\dual$ with $\R$ via the dot product on $\R$. The triple $(\T,\rho,\pi)$ satisfies the assumptions of Theorem \ref{thm:symp_quot_is_symp_DM_stack}. The symplectic Deligne--Mumford stack $\C^d\sq{\pi}\T$ defined by the data $(\T,\rho,\pi)$ is called the weighted projective space $\P(\vec{w})$.

 A stacky polytope $(N,\Delta,\beta)$ giving $\P(\vec{w})$ consists of the following data 
 \begin{itemize}
  \item A $\Z$-module 
	$N = \coker(\dot\rho: \Z \to \Z^d) = \Z^d\big/\Z(w_1,\dots,w_d)$.
  \item A convex polytope 
	\[
	\Delta= \left\{\sum_{\alpha=1}^d s_\alpha \vec{e}^\alpha \in (\R^d)^\dual \!\ \middle| \!\ s_\alpha \geq -c_\alpha \ (\alpha=1,\dots,d),\ \sum_{\alpha=1}^d s_\alpha w_\alpha = 0 \right\}. 
	\]
  \item The natural projection
	$\beta:\Z^d \to N = \Z^d\big/\Z(w_1,\dots,w_d)$.
 \end{itemize}
 Here $N^\dual\otimes\R = (\R^d/\R(w_1,\dots,w_d))^\dual$ is embedded in $(\R^d)^\dual$ via the induced linear map $\dot\rho$ and $c_1,\dots,c_d$ are real constants with $\sum_{\alpha} c_\alpha w_\alpha = \pi$.
\end{examples}

\section{Stacky polytopes versus stacky fans}

We discuss the relation between stacky polytopes and stacky fans in this section. First we review briefly the Deligne--Mumford stack $\X_{\vec{\Sigma}}$ associated to a stacky fan $\vec{\Sigma}$ which is introduced by Borisov et al. \cite{borisov05:_chow_delig_mumfor}. Their construction is an extension of the quotient construction of Cox \cite{cox95}. Next we assign a stacky fan $\vec{\Sigma}$ to a stacky polytope $\vec{\Delta}$ and establish an equivalence between $\X_{\vec{\Sigma}}$ and $\X_{\vec{\Delta}}$. See Cox \cite{cox05:_lectur_toric_variet} for terminology used in the theory of toric varieties.

\begin{dfn}
Consider a triple $(N,\Sigma,\beta)$ of
\begin{itemize}
 \item a finitely generated $\Z$-module $N$ of rank $r$, 
 \item a fan $\Sigma$ with $d$ rays $\rho_1,\dots,\rho_d$ in $\t=N\otimes_\Z\R$,
       and
 \item a homomorphism of $\Z$-modules $\beta :  \Z^d \to N$.
\end{itemize}
We call the triple $(N,\Sigma,\beta)$ a \emph{stacky fan} if the following conditions are satisfied.
\begin{enumerate}
 \item \squashup
       The fan $\Sigma$ is simplicial, that is, the minimal generators of every cone $\sigma \in \Sigma$ are linearly independent in $\t$.
 \item \squashup
       We denote by $\bar{n}_\alpha\ (\alpha=1,\dots,d)$ the image of $n_\alpha = \beta(\vec{e}_\alpha)$ through the natural map $N \to \t$. Then $\bar{n}_\alpha$ generates the ray $\rho_\alpha$ of $\Sigma$.
 \item \squashup
       The cokernel of the homomorphism $\beta :  \Z^d \to N$ is finite.
\end{enumerate}
\end{dfn}

Let $\C[z_1,\dots,z_d]$ be the coordinate ring of $\C^d$. The $\alpha$-th coordinate $z_\alpha$ corresponds to $\alpha$-th ray $\rho_\alpha$. For each cone $\sigma$, the monomial $\prod_{\alpha: \rho_\alpha \not\subset \sigma} z_\alpha$ is denoted by $z^{\hat{\sigma}}$. We define an ideal $J_\Sigma$ of $\C[z_1,\dots,z_d]$ as the ideal generated by the monomials $z^{\hat{\sigma}}\ (\sigma \in \Sigma)$. The Zariski open subset $\C^d\setminus\V(J_\Sigma)$ is denoted by $Z_\Sigma$.

Let $\TC$ be the complex torus $\C/\Z$. Applying the exact functor $\Hom_\Z(-,\TC)$ to the exact sequence \eqref{eqn:exact_sequence_including_DG}, we have 
\[
 \xymatrix@C=40pt{
 G_\C \ar[r]^{\rho_\C} &
 \TC^d \ar[r] &
 T_\C \ar[r] &
 \{\1\}.
 }
\] 
Here $G_\C = \Hom_\Z(\DG(\beta),\TC)$ and $T_\C = \Hom_\Z(N^\dual,\TC)$. Note that we can identify $\Hom_\Z((\Z^d)^\dual,\TC)$ with $\TC^d = \C^d/\Z^d$ naturally. The $d$-dimensional torus $\TC^d$ naturally acts on $Z_\Sigma$. Therefore the group $G_\C$ also acts on $Z_\Sigma$ through the homomorphism $\rho_\C: G_\C \to \TC^d$. Let $\X_{\vec{\Sigma}}$ be the quotient stack $\bigl[Z_\Sigma / G_\C\bigr]$. 

\begin{remark}
 The stack $\X_{\vec{\Sigma}}$ is usually considered as a stack over the category of schemes \cite{borisov05:_chow_delig_mumfor}. Since $Z_\Sigma$ is an open subset of $\C^d$ with respect to the usual topology and the $G_\C$-action on $Z_\Sigma$ is smooth, we consider $\X_{\vec{\Sigma}}$ as a stack over $\Diff$.
\end{remark} 

\begin{prop}
 [Borisov--Chen--Smith \cite{borisov05:_chow_delig_mumfor}]
 For each stacky fan $\vec{\Sigma} = (N,\Sigma,\beta)$, the quotient stack $\X_\vec{\Sigma}$ is a Deligne--Mumford stack. The underlying space of $\X_\vec{\Sigma}$ is the toric variety determined by the fan $\Sigma$. The stack $\X_\vec{\Sigma}$ is called the \emph{toric Deligne--Mumford stack associated to the stacky fan} $\vec{\Sigma}$.
\end{prop}

We can associate a rational fan $\Sigma \subset \t$ to each simple rational polytope $\Delta \subset \t^\dual$ \cite{cox05:_lectur_toric_variet}: each face $F$ of $\Delta$ corresponds to the cone $\sigma_F$ generated by $\bar{n}_\alpha$'s with $F_\alpha \supset F$. Then the set $\Sigma_\Delta = \bigl\{ \sigma_F\!\ \big|\!\ \text{$F$ is a face of $\Delta$} \bigr\}$ is a simplicial fan. Using this correspondence, we can associate the stacky fan $\vec{\Sigma}_\vec{\Delta} = (N,\Sigma_\Delta,\beta)$ to a stacky polytope $\vec{\Delta} = (N,\Delta,\beta)$.

\begin{thm}
 \label{thm:equivalence_between_two_stacks}
 Let $\vec{\Delta}$ be a stacky polytope and $\vec{\Sigma}$ the stacky fan $\vec{\Sigma}_{\vec{\Delta}}$ defined by $\vec{\Delta}$ as above. Then $\X_\vec{\Delta}$ and $\X_\vec{\Sigma}$ are equivalent as stacks over $\Diff$.
\end{thm}

First of all, we note that the natural embedding $\T \to \TC$ induces the commutative diagram
\[
 \xymatrix@C=50pt@R=20pt{
  G_\C \ar[r]^{\rho_\C} & \TC^d \\
  G \ar[r]_{\rho}\ar[u] & \T^d. \ar[u]
 }
\]
Here $G = \Hom_\Z(\DG(\beta),\T)$ as Section~\ref{section:DM_stack_associated_to_stacky_polytope}. The homomorphisms $G \to G_\C$ and $\T^d \to \TC^d$ are both embeddings.

\begin{lem}
 \label{lem:complexification}
 Regarding $G$ as a subgroup of $G_\C$ via the above embedding, we have $G_\C = G \times \exp(\i\g)$. Here $\g = \Lie(G)$.
\end{lem}
\begin{proof}
Let $D_\mathrm{tor}$ be the torsion submodule of the finitely generated $\Z$-module $\DG(\beta)$ and $D_\mathrm{free} = \DG(\beta)/D_\mathrm{tor}$. Since $\DG(\beta) \cong D_\mathrm{free} \oplus D_\mathrm{tor}$, we have 
$G = \Hom_\Z(D_\mathrm{free},\T)\times\Hom_\Z(D_\mathrm{tor},\T)$. Because the abelian group $\Hom_\Z(D_\mathrm{tor},\T)$ is finite, the Lie algebra of $\Hom_\Z(D_\mathrm{free},\T)$ is $\g$. We also have 
\[
 G_\C = \Hom_\Z(D_\mathrm{free},\TC)\times\Hom_\Z(D_\mathrm{tor},\TC).
\]
The lemma follows from $\Hom_\Z(D_\mathrm{tor},\T)=\Hom_\Z(D_\mathrm{tor},\T_\C)$ and $\Hom_\Z(D_\mathrm{free},\TC) = \Hom_\Z(D_\mathrm{free},\T) \times \exp(\i\g)$.
\end{proof}

\begin{lem}
 \label{lem:nonzero_index}
 Define $I(z) = \{ \alpha \!\ |\!\ z_\alpha = 0 \}$ for $z \in \C^d$. Suppose $F$ is a face of $\Delta$ and $\sigma_F$ the cone associated to $F$. Then $z^{\hat{\sigma}_F} \ne 0$ if and only if $F \subset \bigcap_{\alpha \in I(z)} F_\alpha$.
\end{lem}

\begin{proof}
Since $z^{\hat{\sigma}_F} = \prod_{\alpha: \rho_\alpha \not\subset \sigma}z_\alpha$, the monomial $z^{\hat{\sigma}_F}$ is not zero if and only if for all $\alpha$, $\rho_\alpha \not\subset \sigma_F$ implies $z_\alpha \ne 0$. Because of the definition of $\sigma_F$, this is equivalent to the statement that for all $\alpha$, $\alpha \in I(z)$ implies $F \subset F_\alpha$. 
\end{proof}

Consider the family $\F_\tau = \bigl\{I(z) \!\ \big|\!\ z \in \lm\bigr\}$. Using the function (\ref{eq:mu_zero}), we have 
\begin{align*}
 \mu_0\bigl(\lm\bigr) 
 &= \Bigl\{ \textstyle\sum_\alpha \pi|z_\alpha|^2\vec{e}^\alpha \in (\R^d)^\dual \ \Big|\ z \in \C^d,\ \textstyle\sum_\alpha \pi|z_\alpha|^2w^\alpha = \tau \Bigr\} \\
 &= \bigl\{ s \in (\R^d)^\dual \ \big|\!\ \pair{s,\vec{e}_\alpha} \geq 0 \text{ for all } \alpha,\ \rho^\dual(s)=\tau \bigr\}.
\end{align*}
We denote by $\Delta_\tau$ the above set. Then $I \in \F_\tau$ if and only if there is $s \in \Delta_\tau$ such that $I=\{ \alpha\!\ |\!\ s_\alpha=0 \}$.

\begin{lem}
 \label{lem:surjectivity}
 The space $Z_\Sigma$ includes $\lm$. Moreover $Z_\Sigma = \exp(\i\g) \cdot \lm$.
\end{lem}

\begin{proof}
For a set $I \subset \{1,\dots,d\}$, define $O_I = \bigl\{ z \in \C^d \!\ \big|\!\ z_\alpha=0 \text{ if and only if } \alpha \in I\bigr\}$. According to Guillemin--Ginzburg--Karshon \cite[Theorem 5.18]{guillemin02:_momen_hamil}, we have
\[
 \exp(\i\g) \cdot \lm = G_\C \cdot \lm = \bigcup_{I \in \F_\tau}O_I.
\]
Therefore it suffices to show that $Z_\Sigma = \bigcup_{I \in \F_\tau}O_I$.

Suppose $z \in O_I$ for some $I \in \F_\Delta$. Then there exists $s \in \Delta_\tau$ such that $s_\alpha = 0$ if and only if $\alpha \in I(z)$. Since $s \in F = \bigcap_{\alpha \in I(z)} F_\alpha$, $F$ is a nonempty face of $\Delta$. Lemma~\ref{lem:nonzero_index} says that $z^{\hat\sigma_F} \ne 0$. Thus $z \in Z_\Sigma$.

Conversely suppose $z \in Z_\Sigma$. Then there exists a face $F$ of $\Delta$ such that $z^{\hat\sigma_F} \ne 0$. Lemma~\ref{lem:nonzero_index} says that $F \subset \bigcap_{\alpha \in I(z)} F_\alpha$. Since $\bigcap_{\alpha \in I(z)} F_\alpha$ is not empty, we have $I(z) \in \F_\Delta$ and $z \in O_{I(z)}$.
\end{proof}

\begin{lem}
 \label{lem:stabilizer_blongs_to_compact_torus}
 If both $z \in \lm$ and $u \cdot z \in \lm$ hold for $z \in \lm$ and $u \in G_\C$, then $u \in G$.
\end{lem}

\begin{proof}
Lemma~\ref{lem:complexification} says that there are $\xi \in \g$ and $g \in G$ with $u = g\exp(\i\xi)$. Supposing $g\exp(\i\xi) \cdot z \in \lm$, we show that $\xi=0$.

The infinitesimal action of $\xi$ at $z' = (z'_1,\dots,z'_d) \in \C^d$ is given by
\begin{align*}
 \xi_{\C^d}(z')
 &= \bigl(-2\pi\i\pair{w_1,\xi}z'_1,\dots,-2\pi\i\pair{w_d,\xi}z'_d\bigr)
\end{align*}
under the usual identification $T_{z'} \C^d \cong \C^d$. Since a moment map of the $G$-action on $\C^d$ is given by $\mu: \C^d \to \g^\dual$ in \eqref{eq:moment_map_of_G}, the infinitesimal action $\xi_{\C^d}$ is the same as $J\grad \pair{\mu,\xi}$, where $J$ is the complex structure on $\C^d$ and $\grad \pair{\mu,\xi}$ is the gradient vector field of $\pair{\mu,\xi}$ with respect to the standard Riemannian metric on $\C^d$.

Consider the smooth curve $c$ in $\C^d$ defined by
\[
 c: \R \to \C^d;\ \lambda \mapsto g\exp(\i\lambda\xi) \cdot z.
\]
The velocity vector field of the curve is given by
\begin{align*}
 \dot{c}(\lambda)
 &= \bigl(2\pi\pair{w_1,\xi}c_1(\lambda),\dots,2\pi\pair{w_d,\xi}c_d(\lambda)\bigr) 
 = -J \xi_{\C^d}(c(\lambda)) 
 = \grad \pair{\mu,\xi} |_{c(\lambda)}.
\end{align*}
Thus the curve $c$ is an integral curve of $\grad\pair{\mu,\xi}$ and $\pair{\mu(c(\lambda)),\xi}$ is a non-decreasing function in $\lambda$. Since both $c(0) = g \cdot z$ and $c(1) = u \cdot z$ belong to $\lm$, the $1$-form $d\pair{\mu(c(\lambda)),\xi}$ vanishes on $0 \leq \lambda \leq 1$. We have 
\begin{align*}
 \bigl( d\pair{\mu(c(\lambda)),\xi} \bigr) \dot{c}(\lambda)
 &= \omega_0\bigl(\xi_{\C^d}(c(\lambda)),\dot{c}(\lambda)\bigr) 
 =  \bigl( J\dot{c}(\lambda), J\dot{c}(\lambda) \bigr)_{\C^d} 
 =  \bigl\| \dot{c}(\lambda) \bigr\|^2,
\end{align*}
where $(\cdot,\cdot)_{\C^d} = \omega_0(\cdot,J\cdot)$ is the standard Riemannian metric on $\C^d$. The above calculation implies that $\dot{c}(\lambda)=0$ if $0 \leq \lambda \leq 1$. Therefore $ \xi_{\lm}(z) =  \xi_{\C^d}(z) = J\dot{c}(0) = 0$. Since the $G$-action is locally free, we can conclude that $\xi=0$.
\end{proof}

\begin{lem}
 \label{lem:diffeomorphism}
 The map
 \[
 \phi :  \g \times \lm \to Z_\Sigma;\
 (\xi,z) \mapsto \exp(\i\xi) \cdot z
 \]
 is $G_\C$-equivariant diffeomorphism. Here the $G_\C$-action on $\g \times \lm$ is defined by
 \[
 G_\C = G \times \exp(\i\g) \acts \g \times \lm;\
 \bigl(g,\exp(\i\theta)\bigr) \cdot (\xi,z) = (\theta+\xi,g \cdot z).
 \]
\end{lem}

\begin{proof}
According to Lemma~\ref{lem:surjectivity}, the map $\phi$ is well-defined and surjective. Since $G \cap \exp(\i\g) = \{\1\}$, Lemma~\ref{lem:stabilizer_blongs_to_compact_torus} implies that $\phi$ is injective. It is trivial that the map $\phi$ is $\exp(\i\g)$-equivariant. Therefore it suffices to see that the derivative $d\phi(\xi,z)$ is bijective for any $(\xi,z) \in \g \times \lm$. 

For $v \in T_z\lm$ and $\theta \in \g$, take a curve $c:  (-\epsilon,\epsilon) \to \lm$ satisfying $c(0)=z$ and $\dot{c}(0)=v$. According to the proof of Lemma~\ref{lem:stabilizer_blongs_to_compact_torus}, $\exp(\i\g)$-orbits transverse to $\lm$. Therefore the curve
\[
\gamma:  (-\epsilon,\epsilon) \to Z_\Sigma;\ 
\lambda\mapsto \exp(\lambda\i\theta)c(\lambda)
\]
satisfies $\gamma(0)=z$ and $\dot{\gamma}(0)=(\theta,v) \in \t \times T_z\lm \cong T_z Z_\Sigma$. Then we have 
\begin{align*}
 d\phi(\xi,z)(\theta,v)
 &= \left.\dfrac{d}{d\lambda}\right|_{\lambda=0} \phi(\xi+\lambda\theta,c(\lambda)) \\
 &= \left.\dfrac{d}{d\lambda}\right|_{\lambda=0} \exp(\i\xi) \cdot \gamma(\lambda) \\
 &= d\bigl(\exp(\i\xi)\bigr)(\theta,v).
\end{align*}
Therefore $d\phi(\xi,z)$ is bijective.
\end{proof}

Under the diffeomorphism $\phi: \g \times \lm \to Z_\Sigma$, define the $G$-equivariant map $\psi: Z_\Sigma \to \lm$ as the second projection map.

\begin{proof}%
 [Proof of Theorem~\ref{thm:equivalence_between_two_stacks}]
We define a morphism of stacks $\Phi: \X_\vec{\Delta}\to\X_\vec{\Sigma}$ as follows. For an object $U \stackrel{\pi}{\longleftarrow} P \stackrel{\epsilon}{\longrightarrow} \lm$ of $\X_\vec{\Delta}$ over a manifold $U$, the morphism $\Phi$ assigns to it the object
\[
 \xymatrix@C=40pt{
 U &
 P \times_{\lm} Z_\Sigma \ar[l]_(0.6){\Phi_\pi} \ar[r]^(0.6){\Phi_\epsilon} &
 Z_\Sigma
 }
\]
of $\X_\vec{\Sigma}$, where $\Phi_\epsilon:  P \times_{\lm} Z_\Sigma \to Z_\Sigma$ is defined by the fibred square
\[
 \xymatrix@C=50pt{
  P \times_{\lm} Z_\Sigma \ar[r]^(0.6){\Phi_\epsilon} \ar[d] &
  Z_\Sigma \ar[d]^{\psi} \\
  P \ar[r]_(0.45){\epsilon} & 
  \lm,
 }
\]
and $\Phi_\pi$ sends $(p,z)$ to $\pi(p)$. A free right action of $G_\C$ on $P \times_{\lm} Z_\Sigma$ is defined by
\[
 P \times_{\lm} Z_\Sigma \curvearrowleft G_\C;\
 (p,z) \cdot u = (p \cdot \nu(u), u^{-1} \cdot z),
\]
where $\nu: G_\C = G \times \exp(\i\g) \to G$ is the first projection. Then $\Phi_\pi$ is a principal $G_\C$-bundle. 

Let $(f,\tilde{f}): P \to P'$ be an arrow from $P$ to $P'$ in $\X_\vec{\Delta}$. We assign to $\Phi(f,\tilde{f})$ the map $P \times_{\lm} Z_\Sigma \to P' \times_{\lm} Z_\Sigma$ sending $(p,z) \to (\tilde{f}(p),z)$.

It suffices to see that $\Phi$ is a monomorphism and an epimorphism \cite[Proposition 2.1]{behrend08:_differ_stack_gerbes}. 

First we show that the $\Phi$ is a monomorphism. Suppose that we have two objects $P$ and $P'$ of $\X_\vec{\Delta}$ over $U$ and an arrow $\beta$ from $\Phi(P)$ to $\Phi(P')$. The arrow $\beta$ is a $G_\C$-equivariant diffeomorphism which makes the following diagram commute.
\[
 \xymatrix@R=2pt{
   & & P \times_{\lm} Z_\Sigma \ar@<-1ex>[lld]_(0.6){\Phi_{\pi}} \ar@<1ex>[rrd]^(0.6){\Phi_{\epsilon}} \ar@<-4pt>[dd]^{\beta} & & \\
 U & & & & Z_\Sigma. \\
   & & P' \times_{\lm} Z_\Sigma \ar@<1ex>[llu]^(0.6){\Phi_{\pi'}} \ar@<-1ex>[rru]_(0.6){\Phi_{\epsilon'}} & &v
 }
\]
We must to show that there exists uniquely a $G$-equivariant diffeomorphism $\alpha$ from $P$ to $P'$. Since $\psi: Z_\Sigma \to \lm$ is a (trivial) principal $\exp(\i\g)$-bundle, so are $P \times_{\lm} Z_\Sigma \to P$ and $P' \times_{\lm} Z_\Sigma \to P'$. Therefore the $G_\C$-equivariant diffeomorphism $\beta$ induces a $G$-equivariant diffeomorphism $\alpha: P \to P'$. Then $\beta$ is a bundle map over $\alpha$:
\[
 \xymatrix{
  P \times_{\lm} Z_\Sigma \ar[r]^{\beta} \ar[d] & 
  P' \times_{\lm} Z_\Sigma \ar[d] 
  \\
  P \ar[r]_{\alpha} &
  P' 
 }
\]
Since $\Phi_{\epsilon} = \Phi_{\epsilon'} \circ \beta$, we obtain $\beta(p,z)=\bigl(\alpha(p),z\bigr)$. Taking $z \in Z_\Sigma$ with $\epsilon(p)=\psi(z)$, we have
\[
 \epsilon'(\alpha(p))
 = \psi\bigl(\Phi_{\epsilon'}(\beta(p,z))\bigr)
 = \psi\bigl(\Phi_{\epsilon}(p,z)\bigr)
 = \epsilon(p),
\]
and 
\[
 \pi'(\alpha(p)) = \Phi_{\pi'}(\beta(p,z)) = \Phi_{\pi}(p,z) = \pi(p).
\]
Therefore $\beta = \Phi_\alpha$. The arrow $\alpha$ is unique because $\beta$ must be a bundle map over $\alpha$.

Next we show that the $\Phi$ is an epimorphism. Suppose we have an object $\widetilde{P}$ of $\X_\vec{\Sigma}$ over a manifold $U$. Putting $P = \widetilde{P}\big/\!\exp(\i\g)$, we obtain a commutative diagram:
\[
 \xymatrix@R=3pt@C=50pt{
 & \widetilde{P}\ar[dd]\ar[r]^{\tilde{\epsilon}}\ar[ld]_{\tilde{\pi}}
 & Z_\Sigma \ar[dd]^{\psi} \\
 U & & \\
 & P \ar[r]_{\epsilon} \ar[lu]^{\pi} & \lm.
 }
\]
Then $P$ is an object of $\X_\vec{\Delta}$ over $U$ such that $\Phi(P)=\widetilde{P}$.
\end{proof}

\subsection*{Acknowledgements}
The author is grateful to River Chiang, Martin Guest and Reyer Sjamaar for valuable comments. This research is partially supported by NSC grant 98-2115-M-006-006-MY2, and the NCTS (South). This work was supported by National Science Council grant [98-2115-M-006-006-MY2]; and the National Center for Theoretical Sciences (South).

\medskip
{\em
\noindent
Max Planck Institute for Mathematics \newline
Vivatsgasse 7 \newline
53111 Bonn \newline
GERMANY

\noindent
E-mail: sakai@blueskyproject.net
}

\end{document}